\documentclass[11pt]{amsart}
\usepackage{a4wide}
\usepackage{amssymb}
\usepackage{amsmath}

\theoremstyle{plain}
\newtheorem{theorem}{Theorem}[section]
\newtheorem{prop}[theorem]{Prop.}
\newtheorem{lemma}[theorem]{Lemma}
\theoremstyle{definition}
\newtheorem{remark}[theorem]{Remark}
\newtheorem{example}[theorem]{Example}

\def\Slin{S_{\text{\sl lin}}} 
\def\Rlin{R_{\text{\sl lin}}} 
\def\Qlin{Q_{\text{\sl lin}}} 
\def\Dlin{D_{\text{\sl lin}}} 
\def\ddt{{\displaystyle {d\over dt}\Big|_{t=0}}}
\def\Ddt{{\displaystyle {D\over dt}\Big|_{t=0}}}

\def\id{\mathord{\rm id}}
\def\Ad{\mathop{\rm Ad}\nolimits}
\def\inv{^{-1}}
 \def\supp{\mathop{\rm supp}\nolimits}
\def\data{X}
 \def\const{\mathord{\rm const}}
\def\dist{\mathop{\rm dist}\nolimits}
\def\Lie#1{\mathord{\mathfrak #1}}
 \def\tr{\mathop{\rm tr}\nolimits}
\def\gplus{\mathbin{\oplus}}
\def\gminus{\mathbin{\ominus}}
\def\gstar{\mathbin{\setbox0=\hbox{\footnotesize$\bigcirc$}\raise-0.5pt
	\hbox to\wd0{\hfill$*$\hfill}\hskip-\wd0\box0}}
\long\def\nix#1{\relax}
\def\lessim{\lesssim}
\def\wt{\widetilde}
\def\<{\langle}
\def\>{\rangle}
\def\eps{\varepsilon}
 \def\Pt{\mathop{\rm Pt}\nolimits}

\def\level#1{^{\hskip.2ex(#1)}}
\def\idx#1{_{#1}}

\def\allowhyphenation{\penalty1000\hskip0pt}
\def\dash{\discretionary{-}{}{-}\allowhyphenation}

\author{Philipp Grohs}
\address{P.\ Grohs:
	King Abdullah University of Science and Technology, Saudi Arabia,
	and TU Graz, Austria. pgrohs@tugraz.at}
\author{Johannes Wallner}
\address{J.\ Wallner: TU Graz, Austria. Email j.wallner@tugraz.at}
\title
	[Multiscale decompositions for manifold-valued data]
	{Definability and stability of multiscale \\
	decompositions for manifold-valued data.}
\begin{document}

\begin{abstract}
 We discuss multiscale representations of discrete manifold\dash
valued data. As it turns out that we cannot expect general
manifold-analogues of biorthogonal wavelets to possess perfect
reconstruction, we focus our attention on those constructions
which are based on upscaling operators which are either 
interpolating or midpoint\dash interpolating. For definable multiscale
decompositions we obtain a stability result.
\end{abstract}

\maketitle

\setcounter{tocdepth}{2}
\tableofcontents

\newcommand{\Z}{{\mathbb Z}}
\newcommand{\R}{{\mathbb R}}
\newcommand{\schoen}{\mathcal}
\newcommand{\D}{{\schoen D}}
\newcommand{\Ss}{{\schoen S}}
\newcommand{\Rr}{{\schoen R}}
\newcommand{\Q}{{\schoen Q}}

\newpage

\section{Introduction}

\subsection{The problem area.}

The correct multiscale representation of manifold\dash valued data is a basic
question whenever one wishes to eliminate the arbitrariness in choosing
coordinates for such data, and to avoid artifacts caused by applying linear
methods to the ensuing coordinate representations of data.
This question appears to have been proposed first by
D.\ Donoho \cite{donoho-lie}. The detailed paper
\cite{urrahman-2005} describes different constructions, including most
of ours,  and states results
inferred from numerical experiments, but without giving proofs. A series
of papers, starting with \cite{wallner-2005-cca}, has since dealt
with the systematic analysis of upscaling operations on discrete data -- also
known under the name {\em subdivision rules} -- in the case that data live
in Lie groups, Riemannian manifolds, and other nonlinear geometries.
Regarding smoothness of limits, a satisfactory solution has been
achieved by means of the method of 
{\em proximity inequalities} which also play a role in the present paper.
Multiscale decompositions in particular have been investigated by
\cite{grohs-2009-wav} (characterizing smoothness by decay of
detail coefficients) and \cite{grohs-2009-st} (stability).

The present paper studies multiscale decompositions which are analogous
to linear biorthogonal wavelets and reviews the known examples based on
interpolatory and midpoint\dash interpolating subdivision rules including
the simple Haar wavelets. It turns out, however, that it seems unlikely
that a rather general way of defining manifold analogues of linear
constructions can have perfect reconstruction, which is the first
main result of this paper, even if it turns out to be rather vague.
For those multiscale
decompositions which exist, we show a stability
theorem which represents the second main result of the paper.
We further discuss averaging procedures
which work in manifolds equipped with an exponential mapping and which
generalize the well known Riemannian center of mass. This discussion
does not contain substantial new results, but it is included because
we need this construction for the definition of nonlinear up- and downscaling
rules, as well as for converting continuous data to discrete data in the first
place.

\subsection{Biorthogonal wavelets revisited}

We begin by briefly reviewing the notion of biorthogonal Riesz wavelets, 
but we are content with the properties relevant for the
following sections. We start with real-valued sequences 
$\alpha=(\alpha\idx i)_{i\in\Z}$ with
finite support which are called {\em filters} and define
the {\em upscaling rule}, or {\em subdivision rule} associated with the
filter $\alpha$ by 
	$$(S_\alpha c)\idx k
	:= \sum\nolimits_{l\in\Z}\alpha \idx{k-2l} c\idx l.$$ 
 Here $c:\Z\to V$ is any sequence with values in a vector space. 
The transpose of the upscaling rule (we skip the definition of
{\em transpose}) shall be the
{\em downscaling rule} $D$ associated
with the filter $\beta$, via
	$$ (D_\beta c)\idx k
	:=\sum\nolimits_{l\in\Z}\beta \idx{l-2k} c\idx l. $$ 
 Upscaling and downscaling commutes with the 
left shift operator $(Lc)\idx k =c\idx {k+1}$ in the following way:
	$$
	S_\alpha L=L^2 S_\alpha, \quad D_\beta L^2=L D_\beta.
	$$
 The most basic rules are defined by the delta sequence: $S_\delta$
inserts zeros between the elements of the original sequence,
and $D_\delta $ deletes every other element. All rules
can be expressed in terms of $S_\delta$, $D_\delta$, and convolution:
	\begin{align*}
	& S_\delta c=(\dots,c\idx 0,0,c\idx 1,0,c\idx 2,\dots), \quad
	D_\delta c=(\dots,c\idx 0,c\idx 2,c\idx 4,\dots)
	\\ \implies
	&
	S_\alpha c = (S_\delta c) \mathbin* \alpha, \quad
	D_\beta c = D_\delta (c\mathbin* \beta).
	\end{align*}
 We suppress the indices $\alpha,\beta$ from now on.
We assume a further upscaling rule $R$ and a downscaling rule $Q$ which
shall be high pass filters in contrast to low-pass filters
$S$ and $D$.\footnote{usually formulated in terms of Fourier transforms.}

Any sequence $c\level j$, which is interpreted as {\em data at level $j$}
may be recursively decomposed into a 
low-frequency-part $c \level{j-1}$ (data at level $j-1$)
and a high-frequency-part $d\level j$ (details at level $j$) by letting
	\begin{equation}\label{eq:decomp}
	c\level {j-1}=Dc\level j, \quad d\level j=Q c\level j.
	\end{equation}
 This process can be iterated in order to obtain a pyramid consisting of 
{\em coarse data} $c\level 0$ and {\em wavelet coefficients}
$d\level 1,\dots,d\level j$.
Data at level $j$ shall be be reconstructed by
	\begin{equation}
	c\level j = Sc\level {j-1} + R d\level j,
	\end{equation}
 which works precisely if the so-called {\it
quadrature mirror filter equation},
	\begin{equation}\label{eq:qmf}
	SD+RQ=\id,
	\end{equation}
holds. It makes sense to require certain further (`biorthogonality')
properties like
$QR=\id$. In particular, high pass downscaling should annihilate
everything generated by low pass upscaling:
	\begin{equation} \label{eq:biorth}
	QS=0. 
	\end{equation}
 An important consequence of the previous properties is that we can 
rewrite \eqref{eq:decomp} in the form
	\begin{equation}\label{eq:wavalt}
	c\level {j-1} = Dc\level j, \quad
	d\level j = Q (c\level j - S c\level {j-1}).
	\end{equation}
 There are many examples of biorthogonal wavelet 
decompositions. In the following we give some examples.

\subsection{Examples: interpolating and midpoint-interpolating schemes}

\begin{example}  \label{ex:interp-lin}
 An upscaling scheme is called {\em interpolating}, if it keeps
the original data, which is expressed by
	$$
	(Sc)\idx {2k}=c\idx k \iff D_\delta S=\id.
	$$
 For interpolating schemes, downscaling is simply
	$D  = D_\delta$. 
 Then detail coefficients are the difference between data $c$
and the prediction gained via upscaling of $Dc$. With the left shift
operator, we can write
	$$
	Q c = DL(c-SDc). 
	$$
If we define detail coefficients via \eqref{eq:wavalt}, then we can
also employ the modified downscaling operator 
	$$
	Q^{\text{modif}} = DL.
	$$
 Reconstruction works via a basic upscaling rule:
	$$
	R = L^{-1} S_\delta
	$$
It is easy to check that we have indeed perfect reconstruction.
 An example is furnished by the four\dash point scheme
\cite{dyn:1987:4p} defined by $\alpha\idx{\{-3,\dots,3\}}$ $=$
	$(-{1\over 16}$, 
	$0$,
	${9\over 16}$, 
	$1$,
	${9\over 16}$, 
	$0$,
	$-{1\over 16})$.  
 The action $c\level{j-1}=D c\level j$ of the decimation operator
is consistent with
the interpretation of discrete data $c\level j\idx k$ as {\em samples}
of a continuous function $f(t)$ at the parameter value  $t={1\over 2^j}k$.
\end{example}

\begin{example} \label{ex:haar-lin}
 The {\em Haar scheme} is defined by the rules
	\begin{align*}
	S = (L+\id) S_\delta , \quad 
	D = {1\over 2} D_\delta (L+\id) ,\quad
	R = (\id -L) S_\delta, \quad
	Q = {1\over 2}D_\delta (\id-L) ,
	\end{align*}
which operate as follows:
	\begin{align*}
	Sc &= (\dots,c\idx 0,c\idx 0,c\idx 1,c\idx 1,\dots), 
	\\
	Rd & = (\dots, d\idx 0, -d\idx 0, d\idx 1, -d\idx 1,\dots), 
	\\
	Dc &= (\dots,{c\idx 0+c\idx 1\over 2},
		{c\idx 2+c\idx 3\over 2},\dots), 
	\\
	Qc &= (\dots, {c\idx 0-c\idx 1\over 2}, 
		{c\idx 2-c\idx 3\over 2},\dots). 
	\end{align*}
\end{example}

\begin{example} \label{ex:midpt-lin}
 A subdivision scheme $S$ is called {\em midpoint\dash interpolating}, if
it is a right inverse of the decimation operator $D$ which computes
midpoints and which is also used for the Haar wavelets of 
Example \ref{ex:haar-lin}:
	$$
	DS = \id, \quad \mbox{where}\quad
	Dc = (\dots,{c\idx 0+c\idx 1\over 2},
		{c\idx 2+c\idx 3\over 2},\dots).
	$$
 The detail coefficients are the difference between that actual data
$c$ and the imputation $SDc$ found by upscaling the decimated data.
Since $c-SDc$ is by construction in the kernel of $D$ (i.e., 
is an alternating sequence), it contains redundant information. We thus
complete our definitions by letting
	\begin{align*}
	Qc &= D_\delta (c-SDc)  
		= (\dots,(c-SDc)\idx 0, (c-SD c)\idx 2,\dots),
	\\ Rd & =
		(\id-L) S_\delta d = 
		(\dots, d_0,-d_0,d_1,-d_1,\dots).
	\end{align*}
 If we define detail coefficients via \eqref{eq:wavalt}, then  a much
simpler downscaling operator for details can be employed:
	$$
	Q^{\text{modif}} = D_\delta.
	$$
 The action $c\level{j-1}=D c\level j$ of the decimation rule 
is consistent with the interpretation of
discrete data $c\level j\idx k$ as an {\em average} of
continuous data over the interval ${1\over 2^j} \cdot [k,k+1]$. 

The 
defining relation implies that any such $S$ can
be turned into an interpolating subdivision rule $\wt S$ by adding one
round of midpoint computation:
	$$
	\wt S  = {1\over 2}(L+\id)S. 
	$$
 $\wt S$ is interpolatory,
since $D_\delta \wt S = {1\over 2}(D_\delta L + D_\delta ) = DS=\id$.
The relation $S=2(L+\id)^{-1}\wt S$
leads to a way of finding midpoint\dash interpolating
schemes from interpolatory ones, since it can be turned into an 
effective computation by the use of symbols \cite{dyn-2002-ss}. For more
information on that kind of schemes, see e.g.\ \cite{donoho-block}.
\end{example}

\section{Biorthogonal decompositions for manifold-valued data}

\subsection{Manifold analogues of linear elementary constructions.}

The main idea to apply the previous constructions to
manifold\dash valued data is to find replacements for the elementary
operations they are composed of. These are the operations $-$
(``vector is difference of points''),  $+$ (``point plus vector is
a point''), and computing the weighted average of points, which again
yields a point. As to which kind of data are points and which are
vectors, data $c\level j$ at level $j$ shall be manifold-valued sequences
of points, while detail coefficients $d\level j$ shall be sequences with
values in vector spaces associated with the manifold. 

For data with values
in a Lie group $G$, with associated Lie algebra $\Lie g$, we  let
	$$
	p\gplus v := p\exp(v), \quad
	q\gminus p := \log(p^{-1}q)\in\Lie g, 
	$$
 where $\exp$ is the group exponential function and $\log$ is its inverse.
For matrix groups, we have $\exp(x)=\sum_{k\ge 0} x^k/k!$ as usual
(see e.g.\ \cite{bump-2004-lg} for Lie theory).
In a surface or Riemannian manifold $M$, we use the exponential mapping
$\exp_p$ which maps a vector $v$ in the tangent space $T_p M$ to the
endpoint of a geodesic of length $\|v\|$ which emanates from $p$ with
initial tangent vector $v$:
	$$
	p\gplus v := \exp_p(v), \quad
	q\gminus p := \exp_p^{-1}(q) \in T_p M.
	$$
 We have thus found analogues $\gplus$ and $\gminus$ of the $+$ and $-$
operations, respectively. An average with weights of total sum $1$ is in
Euclidean space equivalently definable by 
	\begin{equation}
	\label{mean:elem:equiv}
	m=\sum\alpha_j x_j \iff
	\sum\alpha_j(x_j-m)=0  \iff
	\sum\alpha_j \dist(x_j,m)^2=\min.
	\end{equation}
 The middle definition 
carries over to both Lie groups and Riemannian manifolds
(provided $m$ is unique, which it locally is): 
	\begin{equation}
	\label{eq:def:average0}
	\sum\alpha_j(x_j\gminus m)=0.
	\end{equation}
 In Riemannian manifolds, this average is the same as the one 
defined by the right hand condition.
 These constructions have been employed to define operations on
manifold\dash valued data before, in particular subdivision processes.
For more details the reader is referred to \cite{grohs-2009-st}. 

Another way of redefining averages is by means of an auxiliary
base point: In a vector space, we have
	$$
	\sum\alpha_j = 1\implies \sum \alpha_j x_j = 
	x + \sum\alpha_j(x_j-x),
	$$
 for any choice of $x$.  This leads to the definition
	\begin{equation}
	\label{eq:basepoint}
	x\gplus \Big(\sum\alpha_j(x_j\gminus x)\Big)
	\end{equation}
 of manifold average which involves the choice of an additional base point.

\begin{example} \label{ex:midpoint}
It is not difficult to see that the weights
$\alpha_0=\alpha_1={1\over 2}$ lead to a symmetric average
$m=\mu(x_0,x_1)=x_0\gplus {1\over 2}(x_1\gminus x_0) =
x_1\gplus {1\over 2}(x_0\gminus x_1)$, which can be taken as 
the manifold\dash midpoint of $x_0$ and $x_1$. It fulfills the balance
condition $(x_1\gminus m) + (x_0\gminus m)=0$.
\end{example}

An obvious
generalization, where the averaging process possibly works with a continuum
of values is defined as follows: Consider a set $X$ which
is equipped with some probability measure. For instance we could take the unit interval
$X=[0,1]$ with Lebesgue measure.
The weighted average $m$ of data $(f(t))_{t\in X}$
with values in a vector space is defined by the following equivalent definitions
	\begin{equation}
	m = \int_X f(x) \iff
	\int_X (f(x)-m) = 0 \iff
	\int_X  \dist(f(x),m)^2=\min.
	\end{equation}
 In the case that $X$ is the integers, and the measure means giving each
$i\in\Z$ the weight $\alpha_i$, then this definition reduces to
\eqref{mean:elem:equiv}. Also the integral version of the average can be
made to work for manifold\dash valued data, by defining $m$ via
	\begin{equation}
	\label{eq:def:average1}
	\int_X (f(x)\gminus m) = 0.
	\end{equation}
 In the Riemannian case, which has been thoroughly discussed by
Karcher \cite{karcher-1977-cm}, this is equivalent to 
$\int_X  \dist(f(x),m)^2=\min$. It is then called the Riemannian
center of mass (see Section IX.2 of \cite{kobayashi-69}).

\subsection{Manifold versions of filters.}

 We now define nonlinear analogues of the up- and downscaling rules
$S,D,Q,R$. In order to distinguish them from the corresponding nonlinear
rules, we write the latter as $\Slin$, $\Dlin$, $\Qlin$, $\Rlin$. The 
symbols $\Ss,\D,\Q,\Rr$ denote nonlinear up- and downscaling operators 
which like the linear ones commute with the left shift operator 
in the following way:
	$$
	\Ss L=L^2\Ss, \quad \D L^2=L\D, \quad
	\Rr L=L^2\Rr, \quad \Q L^2=L\Q.
	$$
 We now decompose manifold-valued data `at level $j$', which are denoted
by the symbol $c\level j$ in a manner similar to \eqref{eq:wavalt}:
	\begin{equation}\label{eq:ndecomp} 
	c\level {j-1} = \D c\level j\quad 
	d\level j = \Q \big(c\level j \gminus \Ss \D c\level j\big).
	\end{equation}
 By iteration  we arrive at 
data $c\level 0$ at the coarsest scale together with a pyramid
of detail coefficients $d\level 1,\dots , d\level j$. 
In order to obtain perfect reconstruction via
	\begin{equation}
	c\level j = \Ss c\level{j-1} \gplus \Rr d\level j
	\end{equation}
  we impose the following
condition on the nonlinear operators which could be interpreted as
a nonlinear quadrature mirror filter equation:
	\begin{equation}\label{eq:nqmf}
	\Ss \D c \gplus (\Rr\Q(c\gminus \Ss\D c)) 
	= c\quad \mbox{for all}\ c.
	\end{equation}

\subsection{Examples: interpolating and midpoint-interpolating schemes}

 \begin{example} \label{ex:haar-geom}
 (manifold version of Example \ref{ex:haar-lin})
 We show how the Haar scheme can be made to 
work in groups and in Riemannian manifolds. With the midpoint
$\mu(p,q)$ of Example \ref{ex:midpoint} we let
	\begin{align*}
	\Ss c  &= \Slin c = (\dots,c\idx 0,c\idx 0,c\idx 1,c\idx 1,\dots),
	\\
	\D c  &= (\dots,\mu(c\idx 0,c\idx 1),\mu(c\idx 1,c\idx 2),\dots) 
	\end{align*}
 while $\Q=\Qlin$ and $\Rr=\Rlin$. Indeed,  $c\gminus\Ss\D c$
is an alternating sequence of vectors, and the detail coefficients
associated with data $c$ are given by
	\begin{align*}
	d & =\Qlin(c\gminus\Ss\D c) 
	\\& =
		\Qlin(\dots,
		c\idx 0\gminus \mu(c\idx 0,c\idx 1),
		c\idx 1 \gminus \mu(c\idx 0,c\idx 1),
		c\idx 2\gminus \mu(c\idx 2,c\idx 3),
		\dots),
	\\&= (\dots,
		c\idx 0\gminus \mu(c\idx 0,c\idx 1), 
		c\idx 2\gminus \mu(c\idx 2,c\idx 3), 
		\dots) .
	\end{align*}
 It is obvious that with this definition, $\Ss\D c\gplus \Rr d =c$,
so we have perfect reconstruction.
\end{example}

\begin{example} (manifold version of Example \ref{ex:interp-lin})
To find a nonlinear analogue $\Ss$ of a linear
upscaling rule defined by affine averages, we can employ geometric
averages instead. In this way the interpolating scheme $\Slin=S_\alpha$
can be transferred to the geometric setting, by letting
	$$
	(\Ss c)\idx {2k} = c_k, \quad
	\sum\nolimits_{r\in\Z}\alpha \idx{2r+1} \big(c\idx {k-r}
		\gminus (\Ss c)\idx {2k+1}\big) = 0.
	$$
 The remaining rules can be taken from the linear case (using the
fact that the simplest rules can be applied to
{\em any} sequence, as its elements do not undergo computations).
	$$
	\D=\Dlin = D_\delta, \quad 
	\Q = \Qlin = Q^{\text{modif}} = LD_\delta, \quad
	\Rr = \Rlin = L^{-1} S_\delta.
	$$
 From the interpolating property of $\Ss$ we see that we have perfect
reconstruction.
\end{example}

 \begin{example}  \label{ex:midpt-manif}
(manifold version of Example \ref{ex:midpt-lin})
In order to make a midpoint\dash interpolating rule $\Slin$ work on manifolds,
we define an upscaling operator $\Ss$ which retains the crucial property that
$c\idx k$ is the midpoint of $(\Ss c)\idx {2k}$ and $(\Ss c)\idx {2k+1}$.
For this purpose we use \eqref{eq:basepoint}. We
introduce the following notation for sequences $c, v$ and a point $x\in M$:
	$$
	(c\gminus x)\idx k := c\idx k\gminus x, \quad
	(x\gplus v)\idx k := x\gplus v\idx k,
	$$
 and define
	$$
		(\Ss c)\idx {2k}
	=  
		c\idx k \gplus (\Slin (c\gminus c\idx k))\idx {2k},
	\quad
		(\Ss c)\idx {2k+1}
	=  
		c\idx k \gplus (\Slin (c\gminus c\idx k))\idx {2k+1}
	.
	$$
 It is clear from $(c\gminus c\idx k)\idx k = 0$ and the midpoint\dash
interpolating property of $\Slin$, that $\Ss$ is also midpoint\dash
interpolating:
	$$
		\mu\big((\Ss c)\idx {2k},(\Ss c)\idx{2k+1}\big) = c\idx k.
	$$
 We use the same downscaling operators $\Q$, $\D$  as in the
Haar case of Example \ref{ex:haar-geom}, which yields
	$$
	d\level j\idx k =
	(c\level j\gminus \Ss c\level {j-1})\idx {2k}.
	$$
 By midpoint interpolation, $c\level {j-1}$ and $d\level j$ together
determine the original data $c\level j$: With the geodesic reflection
$\sigma_x(y)$ of $y$ in the point $x$  defined by
	$$
	\sigma_x(y) = x\gplus \big(-(y\gminus x)\big) \quad
	\mbox{or, locally equivalently,} \quad
 	\mu(y,\sigma_x(y))=x,
	$$
we have
	$$
		c\level j\idx{2k} 
	 =
		(\Ss c\level {j-1})\idx {2k}
		\gplus d\level j\idx k,
	\quad 
		c\level j\idx{2k+1} 
	=
		\sigma_{c\level{j-1}\idx k} \big(c\level j\idx{2k}\big).
	$$
 This construction is already contained in \cite{urrahman-2005}.
A nonlinear upscaling operator
$\Rr$ which effects exactly this construction via $c\level j = c \level{j-1}
\gplus \Rr d\level j$ necessarily depends on the data and may be defined by
	$$
	(\Rr d)\idx{2k} = d\idx k ,\quad
	(\Rr d)\idx {2k+1} 
	=
		\sigma_{c\level{j-1}\idx k} \Big( 
		(\Ss c\level {j-1})\idx {2k}
		\gplus d\level j\idx k
		\Big)
		\gminus \Ss c\level j\idx {2k+1}
		.
	$$
 In Riemannian geometry we cannot further simplify this expression.
In the case of matrix groups, we employ the fact that
$\sigma_x(y) = x y\inv x$ and that successive points with
indices $2k,2k+1$ of $\Ss c$ are converted
into each other by geodesic reflection in the point $c\idx k$:
	\begin{align*}
		(\Rr d)\idx {2k+1} 
	&=
		\log\Big[
		\Big(\Ss c\level {j-1}\idx {2k+1}\Big)\inv
		\Big({c\level{j-1}\idx k} \Big)
		\Big(
			\Ss c\level {j-1}\idx {2k} 
			\exp d\level j\idx k
		\Big)\inv
		\Big({c\level{j-1}\idx k} \Big)
		\Big]
	\\ &=
		\log\Big[
		\Big(c\level {j-1}\idx k\Big)\inv
		\Big(\Ss c\level {j-1}\idx {2k}\Big)
		\exp \Big(-d\level j\idx k\Big)
		\Big(\Ss c\level {j-1}\idx {2k} \Big)\inv
		\Big({c\level{j-1}\idx k} \Big)
		\Big]
	\\ &=
		-\Ad_
		{\mbox{\small $(c\level {j-1}\idx k)\inv
		(\Ss c\level {j-1}\idx {2k})$}}
		\big (d\level j\idx k \big)
	=
		-\Ad_
		{\mbox{\small $\exp\big(
			(\Ss c\level {j-1}\idx {2k})
			\gminus
			(c\level {j-1}\idx k) 
		\big)$}}
		\big (d\level j\idx k \big)
	\\&=
		-\Ad_
		{\mbox{\small $\exp\big(
			\Slin(c\level{j-1}\gminus c\level{j-1}\idx k)\idx{2k}
		\big)$}}
		\big (d\level j\idx k \big).
	\end{align*}
 Here we have used the notation $\Ad_g(v)=gvg^{-1}$. Note that in abelian
groups and especially in Euclidean space, where $g\gplus v = g+v$, 
this formula reduces to $\Rr d\idx{2k+1}=-\Rr d\idx {2k}$.
\end{example}

\subsection{On the general feasibility of the construction} 

The examples of geometric and nonlinear multiscale decompositions
given above are special cases, which are based on interpolatory
subdivision rules, or at midpoint\dash interpolating rules.
It is not clear how perfect
reconstruction can be achieved in general. We shall presently see that
there are some basic obstructions which disappear in the linear case.
For simplicity we consider only periodic sequences, because then 
the upscaling and downscaling rules have a finite\dash dimensional
domain of definition. 

\begin{prop} \label{prop:necessary} Smooth rules $\Ss,\D,\Q,\Rr$
can lead to detail coefficients with perfect reconstruction for
periodic data $c\in M^{2n}$ only
if the rank of the mapping $ c\mapsto c\gminus \Ss\D c $ equals
$n\cdot\dim M$, which is half the generic rank of such a mapping.
 \end{prop}

\begin{proof}
 Equation \eqref{eq:nqmf}, which expresses perfect
reconstruction, is equivalent to 
	$$
	\Rr\Q x  = x, \quad
	\mbox{where}\ x=c\gminus \Ss\D c.
	$$
It follows that the mapping
$c\mapsto c\gminus\Ss\D c=\Rr\Q(c\gminus\Ss\D c)$
has rank $\le n\cdot\dim M$, because $\Q$, mapping $2n$ data items
to $n$ detail coefficients, has this property. As to 
the mapping $c\mapsto \Ss\D c$, its rank does not exceed $n\cdot \dim M$,
because $\D$ has this property. In case the
rank is less than 
$n\cdot\dim M$, the mapping $\id_{M^{2n}}: c\mapsto
\Ss\D c \gplus (c\gminus \Ss\D c)$ would
have rank $<2n\cdot\dim M$, a contradiction.
 \end{proof}

The condition of rank $n\cdot\dim M$
which is necessary for perfect reconstruction
as mentioned in Prop.\ \ref{prop:necessary} is unlikely to be satisfied if
both upscaling by $\Ss$ and downscaling by
$\D$ are defined via geometric averaging rules derived from
linear rules $S_\alpha$ and $D_\beta$. The following discussion
of derivatives should make this clear:
We have
	\begin{equation}
	\label{eq:SDdef}
	\sum\nolimits_l \alpha\idx {k-2l}(c\idx l\gminus \Ss c\idx k)=0,
	\quad
	\sum\nolimits_l  \beta \idx{l-2k}( c\idx l\gminus \D c\idx k)=0,
	\end{equation}
 and we are interested in the change in $(\Ss\D c)\idx k$ if each
$c\idx l$ undergoes a 1\dash parameter variation.
We use the abbreviations $\phi$ and $\psi$ for the derivatives of
$\gminus$  with respect to the first and second argument, respectively.
In the Lie group case, where all tangent vectors are represented by
elements of the Lie algebra $\Lie g$, both $\phi$ and $\psi$ are linear
endomorphisms of $\Lie g$. In case of Riemannian manifolds, where
$\gminus:M\times M\to TM$, both $\phi,\psi$ map to 
$T_{p\gminus q}(TM)$. As the next formula shows it is not necessary to
look closer at this abstract tangent space, because we always combine
$\psi^{-1}$ with $\phi$ and the image of $\phi$ occurs only implicitly.
Differentiation of \eqref{eq:SDdef} implies that
	$$
		{d\over dt}(\D c)\idx k  
	=
		-\Big(\sum\nolimits_l\beta\idx {l-2k} 
			\psi_{c\idx l,\D c\idx k}
		\Big)^{-1}
		\Big(\sum\nolimits_l\beta\idx {l-2k} 
			\phi_{c\idx l,\D c\idx k}
			{d\over dt} c_l
		\Big).
	$$
 and further
	\begin{align*}
		{d\over dt}(\Ss\D c)\idx k  
	&=
		\Big(\sum\nolimits_l\alpha\idx {k-2l} 
			\psi_{\D c\idx l,\Ss\D c\idx k}
		\Big)^{-1}
	\\&
		\hphantom{=}\cdot
		\Big(\sum\nolimits_l\alpha\idx {k-2l} 
			\phi_{\D c\idx l,\Ss\D c\idx k}
		\big(\sum\nolimits_r\beta\idx {r-2l} 
			\psi_{c\idx r,\D c\idx l}
		\big)^{-1}
		\big(\sum\nolimits_r\beta\idx {r-2l} 
			\phi_{c\idx r,\D c\idx l}
			{d\over dt} c_r
		\big)
		\Big)
	.	
	\end{align*}
 The precise form of this equation is not relevant, but by observing
that the differentials of $\gminus$ have to be evaluated at {\em many more}
independent locations than the desired rank $n\cdot\dim M$ would suggest,
it is clear that 
only very special filters can lead to rank $n\cdot\dim M$. The situation in the 
linear case is different: The differentials of $\gminus$ are constant,
and the condition that the previous formula defines a mapping of rank
$n$ is an algebraic condition involving the coefficients
of filters $\alpha,\beta$.

Similar considerations show that also the so\dash called log\dash exponential
construction, where a nonlinear rule is constructed via \eqref{eq:basepoint}
(see Ex.\ \ref{ex:midpt-manif}) do not in general yield the rank condition
expressed by  Prop.\ \ref{prop:necessary}.

\section{Stability analysis}

 The point of going through the trouble of decomposing a signal is that 
one expects many detail coefficients $d\level k\idx l$ to be small and 
therefore to be negligible. This is the basis of thresholding in order to
compress data, which makes sense only if one can control 
the change in reconstructed data if we change the
detail coefficients by resetting some of them to zero.
Similarly, quantizing data will result in deviation
from the original. Again, it is important to control that change.
It is the purpose of this section to establish a {\em stability}
result for nonlinear rules which applies to such situations. 

\subsection{Coordinate representations of nonlinear rules}

For the stability analysis we transfer all manifold operations
to a local coordinate chart. This is justified only if the constructions
we are going to analyze are local. The linear upscaling and downscaling
rules defined previously have this property, and so have the nonlinear
ones mentioned in the examples above.

The operators $\gplus$, $\gminus$ are replaced by their respective
coordinate representations, which are denoted by the same
symbols and which are defined in open subsets of
suitable coordinate vector space: We assume that
$\gplus$ maps from $V\times W$ into $V$,
and $\gminus$ maps from $V\times V$ into $W$.
Besides smoothness they are assumed to fulfill the compatibility condition
	\begin{equation}
	\label{eq:compatibility}
	p\gplus (q\gminus p)=q.
	\end{equation}
We further assume that $\gplus$, $\gminus$
are Lipschitz functions, i.e., there exist 
constants $A,B$ with
	\begin{equation}
	A \|p-q\| \le \|p\gminus q\| \le B\|p-q\|.
	\end{equation}
 Locally this is always the case.
Our analysis of stability requires that the operators
$\Ss, \D, \Q, \Rr$ (we do not introduce new symbols for their
coordinate representations)
fulfill some reasonable assumptions which are 
listed below. 
Notation makes use of the symbol ``$\lesssim$'' which
means that there is a uniform constant such that the left hand side
is less than or equal to that constant times the right hand side.
For a sequence $w=(w\idx i)_{i\in\Z}$ we use the notation
$\|w\|:=\sup_{i\in\Z}\|w\idx i\|$.
 
\begin{itemize} \item 
{\it Boundedness of $\Q, \Rr$:}
	The mappings $\Q$, $\Rr$ operate on $W$\dash valued
sequences $w$, which are generated as the difference of point sequences.
They are supposed to satisfy $\|\Q w\|$, $\|\Rr w\|\lesssim \|w\|$, with
respect to some norm $W$ is equipped with.

 \item {\it Reproduction of constants:} For constant data
we require that $\Ss c =c$ and $\D c =c$.

 \item Each of $\Ss,\Rr,\D,\Q$ shall be as smooth as is needed
(in general a little more than $C^1$ will suffice).

\item {\it First-order linearity of $\Ss,\D$ on constant data:} For 
constant sequences  we require that \\[\smallskipamount]
	\begin{minipage}{\linewidth}
	\begin{equation}
	\label{eq:1storder}
	d\Ss\big|_c = \Slin, \quad
	d\D\big |_c =\Dlin 
	\end{equation}
	\end{minipage} \\[\belowdisplayskip]
 for some low-pass upscaling and downscaling operators
$\Slin$, $\Dlin$ operating on $V$\dash valued sequences, and where
$\Slin$ is a convergent
subdivision rule. The only exception shall be Haar case,
where $\Slin=S_\delta $ shall be the splitting rule
(see Ex.\ \ref{ex:interp-lin}). This condition is natural when
one considers $\Ss,\D$ as geometric analogues of linear
constructions which are defined by replacing affine averages
by geometric averages, or by replacing the $+$ and $-$ operations
by $\gplus$ and $\gminus$.

\end{itemize}

\subsection{Stability Results}

The aim of this section is 
to prove the following stability theorem: 

\begin{theorem}\label{thm:stability}
	Suppose that $\Ss, \D, \Q,\Rr$ are upscaling and downscaling
operators which fulfill the nonlinear version \eqref{eq:nqmf}
of the quadrature mirror filter equation, and which also
fulfill the technical conditions listed above.
Consider a data pyramid $(c\level j)_{j\ge 0}$ with
$c\level{j-1}=\D c\level{j}$ which enjoys the weak contractivity property
	\begin{equation}
	\label{eq:cdec}
	\|\Delta c_j\|\lesssim \mu^j \quad (\mu<1).
	\end{equation}
Then the reconstruction procedure of data $c\level j$ at level $j$
from coarse data $c\level 0$ and details 
$d\level 1,\dots , d\level j$ is stable in the sense that there are
constants $D$, $E_1$, $E_2$ such that for all $j$ and any further data
pyramid $\wt c\level i$ with details $\wt d\level i$ we have
	\begin{align}
	\label{eq:close} 
	& \|c\level 0-\wt c\level 0\|\le E_1 ,\quad 
	\|d\level k - \wt d\level k\|\le E_2\mu^{k}\mbox{ for all } k 
	\\ \implies &
	\label{eq:stabestimate} \|c\level j - \wt c\level j\| \le 
	D\big(\|c\level 0-\wt c\level 0\| 
		+\sum\nolimits_{k=1}^j \|d\level k - \wt d\level k\|\big). 
	\end{align}
\end{theorem}

The assumption of decay given by \eqref{eq:cdec} is 
fulfilled for any finite data pyramid (simply 
adjust the constant which is implied by using the symbol ``$\lesssim$'').

\subsection{Proofs}

The remaining part of this section is devoted to the proof of this 
statement. 
Our arguments closely 
follow the ones in \cite{grohs-2009-st} which will enable us to occasionally 
skip over some purely technical details and focus on the main ideas.

The crux is to show that the differentials of the 
reconstruction mappings are uniformly bounded. 
We shall go about this task by using perturbation 
arguments. The justification of this approach lies in the fact that by our 
assumptions the nonlinear reconstruction procedure agrees with a linear 
one up to first order on constant data. Indeed, our assumptions already 
imply that $\Ss$ satisfies a \emph{proximity condition} with $\Slin $ in the 
sense of \cite{wallner-2005-cca}:

 \begin{lemma}\label{lem:prox} 
 With the above 
assumptions we have the inequalities
	\begin{align}
	\|\Ss c - \Slin  c\|  \lesssim \|\Delta c\|^2 ,
	\quad
	\|\D c - \Dlin  c\|  \lesssim  \|\Delta c\|^2.
	\end{align}
 \end{lemma}

 \begin{proof}
 We use a first order Taylor expansion of $\Ss$. For any constant
sequence $e$ we have $\Slin e=Se=e$, so 
	\begin{align*} 
		\Ss c 
	& 
		= \Ss e + d\Ss|_{e}(c-e) + O(\|c-e\|^2)
	\\&
		= e + \Slin (c-e) +O(\|c - e\|^2)
		= \Slin  c + O(\|c-e\|^2).
	\end{align*}
 Since $\Ss$ and $\Slin $ are local operators we may choose $e$ such that
	$$\|c-e\|\lesssim \|\Delta c\|. $$
 This proves the first equation. The proof of the second one is the same. 
\end{proof}

We now show that for all initial data $c\level j$ with exponential decay of
$\|\Delta c\level j\|$, the associated detail coefficients experience
 the same type of decay.

\begin{lemma}\label{lem:direct} Assume that 
 \eqref{eq:cdec} holds for $(c\level j)_{j\geq 0}$. Then
	\begin{equation}
	\label{eq:ddec}
	\|d\level j\|\lesssim \mu^j. 
	\end{equation}
 \end{lemma}

 \begin{proof} We use the boundedness of $\Q$ and Lemma \ref{lem:prox} to 
 estimate the norm of detail coefficients:
	\begin{align*} 
		\|d\level j\| 
	&=
			\|\Q c\level j \gminus \Ss c\level{j-1}\| 
		\lesssim \|c\level j\gminus \Ss c\level{j-1}\|
		\lesssim \|c\level j - \Ss c\level{j-1}\| 
	\\&
		\le	\|c\level j - \Slin  \D c\level{j}\|+\|\Slin  
		c\level{j-1} -\Ss c\level{j-1}\|
	\\ &	\lesssim 
		\|c\level j - \Slin  \Dlin  
		c\level j\| + \|\Slin  (\D c\level j - \Dlin c\level j)\| +\|\Slin  c\level{j-1} -\Ss 
		c\level{j-1}\| 
	\\&
		\lesssim
		\|c\level j - \Slin  \Dlin  c\level j\| + \mu^{2j}.
	\end{align*}
 It remains to estimate $\|c\level j - \Slin  \Dlin  c\level j\|$.
Reproduction of constants implies that
for any constant sequence $e$,
	$$
		\|c\level j - \Slin  \Dlin  c\level j\|
	=
		\|c\level j - e
		+ \Slin  \Dlin ( c\level j- e)\|
	\lesssim \|c\level j - e\|. 
	$$
 By the locality of $\Slin $ and $\Dlin $ we can pick $e$ such that 
$\|c\level j - e\|\lesssim \|\Delta c\level j\|$. This concludes the proof. 
\end{proof}

For later use we record the following two facts. The first one is a 
perturbation theorem which has been shown in \cite{wallner-2005-cca}.

\begin{theorem}\label{thm:proxcon} Assume that $\Slin $ is a convergent 
linear subdivision scheme and that $\Ss$ satisfies $d\Ss|_c = \Slin $ for all 
constant data $c$. Then there exists $\mu < 1$ such that
	\begin{equation}\label{eq:ncontr} \|\Delta \Ss^j c\| \lesssim 
	\mu^j \end{equation}
 for all initial data $c$ with $\|\Delta c\|$ small enough.
\end{theorem}

We do not want to go into details concerning the precise
meaning of `small enough'. The reader who is 
interested in the considerable technical subtleties arising from this 
restriction and also the fact that $\Ss$ is usually not globally defined 
is referred to our previous work \cite{grohs-2008-sbg,grohs-2009-st,grohs-2009-wav} 
where these issues are rigorously taken into account and the appropriate
bounds for $\|\Delta c\|$ are derived.

The second result is also a perturbation result which has been shown in
\cite{grohs-2009-st}.

 \begin{lemma}\label{lem:perturb} 
 Let $A_i$, $U_i$ be operators on a normed vector space. Assume 
exponential decay $\|U_i\|\lessim \mu^i$, for some $\mu<1$.
Then uniform boundedness of
	$\|A_1\cdots A_k\|$ implies  uniform boundedness of
	$\|(A_1+U_1)\cdots(A_k+U_k)\|$.
\end{lemma}

\nix{Next we show that the decay property $\|\Delta c_j\| \lesssim \mu^j$ still 
holds for perturbed data $\wt c_j = P_j(\wt c_0 , \wt d_1,\dots , 
\wt d_j)$, provided that $\|\wt c_0 - c_0\|, \|\wt d_1 - 
d_1\|,\dots ,\|\wt d_j - d_j\|$ are small.}

We continue with the proof of Theorem \ref{thm:stability} by showing that
the decay property \eqref{eq:cdec} we assumed for the data pyramid
$c\level j$ also holds for the perturbed data pyramid $\wt c\level j$.

\begin{lemma}\label{lem:lip}
 Under the assumptions of Theorem \ref{thm:stability}, further assume
that $\Slin $ is a convergent subdivision scheme. Then there exist constants
$s_1$, $s_2$ such that for all $j$, and any choice of data $\wt c\level j$
we have
	\begin{equation}
	\|\Delta \wt c\level 0\|\le s_1 ,\quad
	\|\wt d\level k \|\le s_2\mu^{k}\ \mbox{for all}\ k
	\quad \implies \quad
	\|\Delta \wt c\level j\|\lesssim (\mu+\eps)^j. 
	\end{equation}
 Here for each $\eps >0$ the implied constant is uniform.
 \end{lemma}

	\begin{proof} (Sketch)
 We make the simplifying assumption that for all initial data $c$ which
occur in the course of the proof we have
	 \begin{equation}\label{eq:simpdec} 
	\|\Delta \Ss c\| \le \mu\|\Delta c\|. 
	\end{equation}
 This is no big restriction as it can be shown that such an equation 
always holds for some iterate $\Ss^N$ of $\Ss$ and initial data
with $\|\Delta c\|$ small enough, provided $\Slin $ is 
convergent \cite{wallner-2005-cca}. In case that only 
	$$
	\|\Delta \Ss c\| \le \bar\mu\|\Delta c\|
	$$
 for some $\bar\mu\in(\mu,1)$ we make the 
initial $\mu$ larger. This does not change the substance of Theorem
\ref{thm:stability}. With the Lipschitz constants $r,r'$ defined by
	$ \|\Rr c\| \le r \|c \|$,
	$ \|a\gplus b-a\| \le r' \|a\gplus b\|$
we now estimate:
	\begin{align*}
		\|\Delta \wt c\level 1\|
	& \le
		\|\Delta \Ss \wt c\level 0\|
		+2\|(\Ss \wt c\level 0 \gplus \Rr \wt d\level 1 )
		-\Ss \wt c\level 0 \|
	\\ 
	&\le
		\mu\| \Delta \wt c\level 0\| 
		+ 2 r' \|\Rr \wt d_1\|
	\le
		\mu s_1 + 2 rr' s_2 \mu
	\end{align*}
Iteration of this argument gives the inequality
	$$
	\|\Delta \wt c\level n\|
	\le
	 s_1\mu^n + 2n rr' s_2 \mu^n
	\lesssim (\mu+\eps)^n $$
 for all $\eps >0$, which we wanted to show. In case \eqref{eq:simpdec}
does not hold for $\Ss$, but only for an iterate $\Ss^N$, a similar
argument is required which we would like to skip.
The reason for requiring $s_1,s_2$ to be `small enough' is
that \eqref{eq:simpdec} usually only holds for data $c$ in 
some set
	$$P_{M,\delta}:=\{c\mid c\idx k \in M\ \forall k,\mbox{ and }
	\|\Delta c\|<\delta \}.
	$$
 In general we need to ensure that all $c\level i$'s lie in the set
$P_{M,\delta}$ if the only information on the data is the size of
detail coefficients. This 
rather technical step is where we the restrictions on the constants $s_1,s_2$ 
come in. We chose to skip the technical details regarding this issue, 
since we do not find them particularly enlightening and they have already 
been treated in full detail in previous work
\cite{grohs-2009-st,grohs-2009-wav,grohs-2008-sbg}.
 \end{proof}

We are finally in a position to prove Theorem \ref{thm:stability}. 

\begin{proof}[Proof (of Theorem \ref{thm:stability})] 
The mapping which computes data $c\level k$ at level $k$ by
way of reconstruction is denoted by $P_k$. We use the following
notation and definition:
	\begin{align}
	\data_j & := (c\level 0, d\level 1,\dots , d\level k)
		\in \ell(V)\times\ell(W)^k
	\\
	\label{eq:iter}
		P_k(\data_k) 
	&:= 
		\Ss P_{k-1}(\data_{k-1}) \gplus \Rr d\level k,
		\quad
		P_0=\id.
	\end{align}
 We first treat the case that $\Slin $ is a convergent subdivision scheme and
later deal with the Haar case. 

Observe that we can without loss of generality
assume that both $\|\Delta c\level 0\|$ and the implied constant in 
\eqref{eq:ddec} are arbitrarily small. 
This is because we can simply do a re\dash indexing $(c')\level i=c\level{i+j_0}$
and we assumed exponential decay of $\Delta c\level j$.
in particular,
	$$
	\|\Delta c\level 0\| \le f_1<s_1, \quad 
	\|d\level k\|\le f_2\mu^k, \quad f_2<s_2,
	$$
  with the constants $s_1,s_2$ from Lemma \ref{lem:lip}. 
By Lemma \ref{lem:direct},
$\|d\level j\|$ is likewise of exponential decay. By the same argument
we can make the implied constant arbitrarily small.

Pick the constants $E_1, E_2$ such that
$f_1 + E_1 \le s_1$ and $f_2 + E_2 \le s_2$, and consider
coarse data $\wt c\level 0$ and detail coefficients
$\wt d\level 1,\dots,\wt d\level j$ which obey the
assumption \eqref{eq:close} made in the statement of the theorem. 
Lemma \ref{lem:lip} implies that we have exponential decay of 
$\|\Delta \wt c\level j\|$. 

The estimates gathered so far enable us to show 
that there exists a constant $C$ such that for all $j$, $k$ and all 
perturbed arguments
	$$
	\wt\data_j 
	= (\wt c\level 0, \wt d\level 1 , \dots , \wt d\level j),
	$$
we have the bound
	\begin{equation}
	\label{eq:diffbound}
	\Big\|\frac{\partial}{\partial d\level k}
		\Big|_{\wt\data_j} P_j\Big\|,  \quad
	\Big\|\frac{\partial}{\partial c\level 0}
		\Big|_{\wt\data_j} P_j\Big\| \le C.
	\end{equation}
 Indeed, using the chain rule on the recursive definition \eqref{eq:iter},
we see that
	 \begin{equation}\label{eq:norm}
	\frac{\partial}{\partial c\level 0}
		\Big|_{\wt\data_j} P_j
	= 
		\Big({d_1\oplus}
			\big|_{(\Ss P_{j-1},\Rr d\level j)}\Big)
		\Big(d\Ss
			\big|_{\wt c\level{j-1}}\Big)
	\Big(\frac{\partial}{\partial c\level 0}
		\Big|_{\wt\data_{j-1}} P_{j-1}\Big).
	\end{equation}
 Our assumptions on smoothness (here: $\gplus$ is $C^2$) and
the compatibility relation \eqref{eq:compatibility} together imply that
	 \begin{align*}
		d_1\mathord\oplus\big|_{(\Ss \wt c\level {j-1},\Rr d\level j)}
	&=
		d_1\mathord\oplus \big|_{(\Ss \wt c\level{j-1},0)}
	+	\Big(
		d_1\mathord\oplus \big|_{(\Ss \wt c\level{j-1},\Rr d\level j)}
	-	d_1\mathord\oplus \big|_{(\Ss \wt c\level{j-1},0)}
		\Big) = I + V_j
	\end{align*}
 with $\|V_j\|\lesssim \|\Rr d\level j\|\lesssim \mu^j$.
In order to estimate the term $d\Ss\big|_{\wt c\level{j-1}}$,
we note that $d\Ss = \Slin $ implies that
$\|d\Ss\big|_c-\Slin \|\lesssim \|\Delta c\|$ for all initial data $c$,
see \cite{grohs-2009-st}. Hence we can write 
	$$d\Ss\big|_{\wt c\level{j-1}}=\Slin  + W_j,
	\quad\mbox{where}\ \|W_j\|\lesssim \|\wt c\level{j-1}\|
	\lesssim (\mu+\eps)^{j}$$
 for any $\eps >0$. It is 
a well known fact that for a convergent subdivision scheme $\Slin $, there 
is a constant $M$ with $\sup_j \|\Slin ^j\|\le M$. The previous 
discussion and iterative application of \eqref{eq:norm} implies 
	$$
	\frac{\partial}{\partial c\level 0}
		\Big|_{\wt\data_j}
		P_j =(\Slin  + U_1)\cdots(\Slin+U_j), \quad
	\mbox{where}\
 	\|U_k\|\lesssim (\mu+\eps)^k.
	 $$
 Now we invoke Lemma \ref{lem:perturb} and see that indeed the partial
derivatives of $P_k$ with respect to $\wt c\level 0$
at $\wt\data_j$ are uniformly bounded, independent of $j$.  The
derivatives with respect to $\wt d^k$ can be handled in an analogous
manner. This shows \eqref{eq:diffbound}, from which it is easy to
see \eqref{eq:stabestimate}. 

Having concluded the proof in the case that $\Slin$ is a convergent
subdivision scheme, we turn to the Haar case. It is analogous, but 
because we have $\Ss=\Slin$ we do not need the perturbation inequalities
at all to estimate differentials (in particular we do not need
Lemma \ref{lem:lip}). 
 \end{proof}

\begin{remark} The only place where the constants $E_1,E_2$ come into 
play is the assumption \eqref{eq:simpdec} which is usually only satisfied 
for data in some set $P_{M,\delta}$ -- see the discussion in the proof
of Lemma \ref{lem:lip}. It is easy to see that if $\Ss$ is 
defined and contractive for {\em all} initial data, then the
constants $E_1$, $E_2$ can be arbitrarily large.
 \end{remark}

\section{Obtaining discrete data} 
\subsection{Convolution and smoothing of manifold\dash valued data}

Here we are going to investigate further
properties of the geometric average which was defined by Equations
\eqref{eq:def:average0} and \eqref{eq:def:average1}. They 
will become important in Section \ref{sec:last}. This material 
is already contained in Karcher's paper \cite{karcher-1977-cm} as far as 
surfaces and Riemannian geometry are concerned. Here we also show the
extension to Lie groups, which is not difficult once the Riemannian
case is known.

Convolution with a function $\psi$ with $\int\psi=1$ can be interpreted
as an average. This applies to multivariate functions as well as to 
univariate ones, which are our main concern.
In order to fit the previous definitions, we give
an equivalent construction of the convolution $g\mathbin *\psi$ for vector\dash valued
functions $g$, and at the same time a definition of
$(f\gstar \psi)(u)$ for manifold\dash valued
functions $f:\R^d\to M$. 
	\begin{align}
	&\textstyle 
		m
	=
		(g\mathbin*\psi)(u) \iff m=\int_{\R^d} g(x)\psi(u-x)\, dx 
	\iff
		\int_{\R^d} (g(x)-m)\psi(u-x)\, dx = 0,
	\\&\textstyle 
		m
	=	(f\gstar \psi)(u) 
	\iff
		\int_{\R^d} (f \gminus m) \psi(u-x) dx =0.
	\end{align}
 The even
more general case where the domain of functions are manifolds has been
discussed in \cite{karcher-1977-cm}. It turns out that basically 
any nonnegative kernel function $\psi$ supported in the cube $[-1,1]^d$
can be used for smoothing in the following way: For each $\rho>0$,
we let
	\begin{equation}
	f^{\rho}  =  f \gstar \psi^{\rho}, 
	\quad
	\mbox{where} \
	\psi^{\rho}(x) = {1\over\rho^d} \psi\Big({x\over \rho}\Big).
	\end{equation}
 We want to show that $f$ and its differential $df$ are approximated
by $f^{\rho}$ and $df^{\rho}$ as  $\rho$ approaches zero.
The proofs consist of revisiting the proofs given in \cite{karcher-1977-cm}
which apply to the Riemannian case.

 \begin{theorem} \label{thm:conv}
 Consider the smoothed functions $f^\rho$ defined by a function
$f:\R^d\to M$ and a kernel $\psi$ as above. Then 
	$$
	\lim\nolimits_{\rho\to 0} f^\rho = f, \quad
	\lim\nolimits_{\rho\to 0} df^{\rho} = df.
	$$
 In case $f$ is Lipschitz differentiable, then this convergence
is linear.
\end{theorem}

\begin{proof} We skip convergence of $f^\rho$ and show only convergence
of $df^\rho$. The proof is in the spirit of Lemma 4.2 and Theorem 4.4
of \cite{karcher-1977-cm}, the difference being that the domain of $f$ is a
vector space. We define $V:\R^d\times M\to \R^{\dim M}$ by letting 
	$$
	\textstyle
	V(u,p) := \int (f(x)\gminus p) \psi^{\rho}(u-x)  dx.
	$$
 By definition, $V(u,f^\rho(u))=0$. This implies the following
equation of derivatives:
	\begin{equation}
	\label{sumderiv}
	d_1 V_{u,f^{\rho}(u)} 
	+ D_2 V_{u,f^{\rho}(u)} \circ df^{\rho}_u  = 0.
	\end{equation}
 The capital $D$ indicates the fact that in the Riemannian case we
employ a covariant derivative. The partial derivatives of $V$ have the form
	\begin{align*}
		d_1\big|_{u,p} V(\dot u)
	&=
		\ddt \textstyle
		\int
			\big(f(y) \gminus p\big) 	
			\psi^{\rho}(u(t)-y)
			 dy
	\\
	&=
		\ddt \textstyle
		\int
			\big(f(x-u+u(t)) \gminus p\big)
                        \psi^{\rho}(u-x)
                         dx
		,
	\\
		D_2\big|_{u,p} V(\dot p) 
	& = 
		\Ddt\textstyle
		\int
			\big(f(x) \gminus p(t)\big) 	
			\psi^{\rho}(u-x)
			 dx
	\end{align*}
 Using the functions $E_{p,q}(\dot q) = -{D\over dt} (p\gminus q(t))$ and 
$F_{p,q}(\dot p) = {d\over dt} (p(t)\gminus q)$, we get
	\begin{equation}
	\label{eq:Vexplizit}
		d_1\big|_{u,p} V(\dot u)
		+ D_2\big|_{u,p} V(\dot p)
	= \int 
		\big(
			F_{f(x),p}(df_x(\dot u))
			-E_{f(x),p}(\dot p)
		\big)
		\psi^{\rho}(u-x)
		 dx.
	\end{equation}
 It is shown in \cite{karcher-1977-cm} that in the Riemannian case the
functions $E_{p,q}$ and $F_{p,q}$ can be bounded in terms of
sectional curvature $K$, and the parallel transport operator
$\Pt_{\text{\sl from}}^{\text{\sl to}}$:
	$$
		E_{p,q}(v) = v + R,
	\quad
		F_{p,q}(v) = \Pt_p^q(v) + R',
	$$
 where $\|R\| \le   \|v\| \const(\min K,\max K) \cdot \dist(p,q)^2$
and  $\|R'\| \le \|F_{p,q}(v)\|\const(\max|K|)\cdot\dist(p,q)^2$.
 Letting $p=f^{\rho}(u)$ and $\dot p = df^{\rho}(\dot u)$, we convert
\eqref{sumderiv} and \eqref{eq:Vexplizit} into the integral
	\begin{align*}
	0 &
	= 
		\int
		\big(
			\Pt_{f(x)}^{f^{\rho}(u)} df_x(\dot u)
			+ R'(x)
			-df^{\rho}(\dot u) 
			- R(x)
		\big)
		\psi^{\rho}(u-x)
		 dx,
	\end{align*}
 without indicating the dependence of the remainder terms $R$, $R'$ on $x$.
 The assumption that $f$ is $C^1$ implies that for all $x$ with 
contribute to the integral (i.e., $\psi^{\rho}(u-x)\ne 0$), we have
	$ x\to u$,
	$df_x(\dot u)\to df_u(\dot u)$,
	$f^{\rho}(x) \to f(u)$, 
	$\Pt \to \id$, 
	$R\to 0$,
	$R'\to 0$.
Observe that all these limits have at least linear convergence rate,
provided $df$ is Lipschitz. With $\int\psi=1$, we obtain
	$$
	\lim\nolimits_{\rho\to 0} (df_x+df^{\rho})(\dot u)\to 0,
	$$
 where the limit is linear if $df$ is Lipschitz. This concludes the proof
in the Riemannian case. 

In the Lie group case, it is not difficult 
the compute the derivatives $E_{p,q}(\dot p)$ and $F_{p,q}(\dot q)$ 
by means of the Baker\dash Campbell\dash Hausdorff formula which says
	$
	\log(e^x e^y) = x + y + {1\over 2}[x,y] +  \cdots,
	$
where the dots indicate terms of third and higher order expressible
by Lie brackets. When $p$ and $q=pe^z$ undergo 1\dash parameter variations
of the form $p(t)=p e^{tw}$ and $q(t)=qe^{tw}$ with $w\in\Lie g$, then 
	\begin{align*}
		p(t)\gminus q 
	& =
		\log(e^z e^{tw}) = z + tw + {1\over 2}[z,tw] + \dots
	\\
	 	p\gminus q(t) 
	& =
		\log(e^{-tw} e^z) = -tw+z+{1\over 2}[-tw,z]+\dots
	\end{align*}
This implies
	\begin{align*}
		F_{p,q}(w) &=w+{1\over 2}[z,w] + \cdots 
	,\\
		E_{p,q}(w) &= w + {1\over 2}[w,z] + \cdots
	\end{align*}
 Similar to the Riemannian case above, we convert
 \eqref{sumderiv} and \eqref{eq:Vexplizit} into the integral
	\begin{align*}
	& \int 
		\Big(
			df_x(\dot u) 
	+ 	{1\over 2}\Big[f^\rho(u)\gminus f(x), df_x(\dot u)
		+df^\rho(u)\Big]
	-	df^\rho(u)
	+	\cdots
		\Big)
		\psi^{\rho}(u-x)
		\, dx
	=0
	\end{align*}
 in the Lie algebra.
 The same arguments imply $x\to u$, $f^\rho(x)\to f(u)$,
$df_x(\dot u)\to df_u(\dot u)$, and as a consequence
$df^\rho\to df$ as $\rho\to 0$.
This concludes the proof of Theorem \ref{thm:conv} in the Lie group case.
 \end{proof}

\subsection{The passage from continuous to discrete data}
\label{sec:last}

In the analysis of multiscale decompositions one frequently assumes
an infinite detail pyramid. In practice a vector\dash valued or
manifold\dash valued function $f(t)$ which depends on a parameter $t\in\R$
is given be finitely  many measurements. Such measurements might
be samples at parameters $t_i = ih$, for some small $h$; or measurements
might be modeled as averages
of the form $f\gstar \phi(\cdot\, -ih)_{i\in\Z} $ where
$\phi$ is some kernel with $\int\phi=1$ and $\supp(\phi)$ small (in fact
physics excludes the kind of measurement we called {\em samples} and
permits only $\phi$ to approach the Dirac delta).

In the linear case any multiscale decomposition based on 
midpoint\dash interpolation and especially the Haar scheme are well adapted
to deal with averages: The decimation operator $D$ in this case
is consistent with the definition of discrete data as follows:
	\begin{align*}
	\psi\level j & = 2^j 1_{[0,1]} (2^j\, \cdot\,) = 2^j 1_{[0,2^{-j}]}
	,\quad
		f\level j = f\mathbin *\psi\level j,
	\quad
		c\level j = f\level j\big|_{2^{-j}\Z}
	\\ \implies
	c\level {j-1} & = D c\level j.
	\end{align*}
 We have no analogous relation for manifold\dash valued multiscale
decompositions. Nevertheless we may let
	$$
	f\level j = f\gstar \psi\level j, \quad
	c\level j = f\level j\big|_{2^{-j}\Z}.
	$$
 In view of Theorem \ref{thm:conv}, this yields discrete data whose
discrete derivatives $\Delta c\level j$ approximate 
the derivatives of $f$. Assuming $f$ to be $C^2$, we have
	\begin{align*}
		\Delta c\level j\idx k 
	&:=
		2^j\big (c\level j\idx {k+1}  - c\level j\idx k  \big)
	\implies
	\Delta c\level j\idx k
	= 
		 {d\over dt} f\level j \big|_{k2^{-j}}
		+O(2^{-j})
	= {d\over dt} f\big|_{k2^{-j}} + O(2^{-j}).
	\end{align*}
 The previous equation is to be interpreted in any
smooth coordinate chart of the manifold under consideration.

\bigskip\centerline{\sc Acknowledgments}\bigskip

The authors gratefully acknowledge the support of the Austrian Science 
Fund. The work of Philipp Grohs has been supported by
grant No.\ P19780.

  \def \http#1{{\it\spaceskip 0 pt plus 0.5pt http:/$\!$/ #1}}
\providecommand{\bysame}{\leavevmode\hbox to3em{\hrulefill}\thinspace}
\providecommand{\MR}{\relax\ifhmode\unskip\space\fi MR }
\providecommand{\MRhref}[2]{%
  \href{http://www.ams.org/mathscinet-getitem?mr=#1}{#2}
}
\providecommand{\href}[2]{#2}

\end{document}